\documentclass[11pt]{article}

\usepackage{graphicx}
\usepackage{amsmath}
\usepackage{amsfonts}
\usepackage{amssymb}
\usepackage{amsthm}
\usepackage{caption}
\usepackage[nooneline, figbotcap]{subfigure}
\usepackage{mathrsfs} 
\usepackage{multirow}
\usepackage{hyperref}
\usepackage{enumerate}

\newtheorem{theorem}{Theorem}

\newtheorem{corollary}[theorem]{Corollary}

\newtheorem{lemma}[theorem]{Lemma}

\subfiguretopcaptrue

\renewcommand{\S}{\Sigma }

\newcommand{\D}{\mathscr{D}}
\newcommand{\Z}{\mathbb{Z}}
\newcommand{\bb}{\beta}
\renewcommand{\aa}{\alpha}
\newcommand{\cc}{\gamma}
\newcommand{\tc}{\tau\gamma}
\newcommand{\ta}{\tau\alpha}
\newcommand{\tb}{\tau\beta}

\newcommand{\pD}{\pi(\D)}
\newcommand{\aD}{\mathscr{A}(\D)}
\newcommand{\MM}{M_1\#M_2}

\begin{document}

\title{On the subadditivity of Montesinos complexity of closed orientable
3-manifolds}

\author{\'Alvaro Lozano, Rub\'en Vigara \\
alozano@unizar.es, rvigara@unizar.es\\
+34 976 739 856, +34 976 739 769\\
Centro Universitario de la Defensa - Zaragoza\\
I.U.M.A. - Universidad de Zaragoza}

\maketitle

\begin{abstract}
A filling Dehn sphere $\Sigma$ in a closed 3-manifold $M$ is
a sphere transversely immersed in $M$ that defines a cell decomposition of
$M$. Every closed 3-manifold has a filling
Dehn sphere \cite{Montesinos}. The Montesinos complexity of a $3$-manifold $M$
is defined as the minimal number of triple points among all the filling Dehn
spheres of $M$. A sharp upper bound for the Montesinos complexity of the
connected sum of two 3-manifolds is given.

\emph{Keywords}: 3-manifold, immersed surface, filling Dehn sphere, triple
points, complexity of 3-manifolds.

Math. Subject Classification: 57N10, 57N35
\end{abstract}

\section{Introduction}\label{sec:Introduction}

Throughout the paper all $3$-manifolds are assumed
to be closed, that is, compact, connected and without boundary, and orientable.

Let $M$ be a $3$-manifold.

A \emph{Dehn sphere} in $M$ is $2$-sphere transversely immersed in $M$, and thus having only double point and triple point singularities. A Dehn sphere in $M$ is \emph{filling} if it naturally
defines a cell decomposition of $M$ (see Section
\ref{sec:Dehn-spheres-Johansson-diagrams} for details).  Following
\cite{Haken1}, in \cite{Montesinos} it is proved that every closed, orientable
3-manifold has a filling Dehn sphere, and filling Dehn spheres and their
Johansson diagrams are proposed as a suitable way for representing all closed,
orientable 3-manifolds. A weaker version of filling Dehn spheres are the so
called \emph{quasi-filling Dehn spheres} in the notation introduced in
\cite{Amendola02}. A quasi-filling Dehn sphere in $M$ is a
Dehn sphere whose complementary set in $M$ is a disjoint union of open
$3$-balls. In \cite{FennRourke} it is proved that every $3$-manifold has a
quasi-filling Dehn sphere.

A simple check using Euler characteristics shows that the number of triple
points of a Dehn sphere in $M$ is always even. The filling Dehn sphere $\S$ in
$M$ is \emph{minimal} if there is no 
filling Dehn sphere in $M$ with less triple points than $\S$. We define the 
\emph{Montesinos complexity} of $M$, $mc(M)$, as the number of triple points of 
a minimal filling Dehn sphere of $M$.

Montesinos complexity has been introduced 
in~\cite{RHomotopies} with a different name (see Section \ref{sec:comments}),
and it is closely related with G. Amendola's 
\emph{surface-complexity $sc(M)$} introduced in~\cite{Amendola02}.

Surface-complexity is subadditive under connected sums, that is, 
\[
  sc(M_1\# M_2) \leq sc(M_1)+sc(M_2),
\]
were $M_1\# M_2$ denotes the connected sum of the $3$-manifolds $M_1$ and $M_2$.

Unlike in the previous case, Montesinos complexity is not subadditive.

\begin{theorem}
  \label{thm:subadditivity-2}
  For any $3$-manifolds $M_{1}$ and $M_{2}$ we have
  \[
    mc(M_1\#M_2)\leq mc(M_1)+mc(M_2)+2\,. 
  \]
\end{theorem}
The aim of this paper is to prove Theorem \ref{thm:subadditivity-2}, and that
the upper bound given there is sharp.

The proof of Theorem~\ref{thm:subadditivity-2} relies on a surgery operation, 
similar to the one developed in~\cite{RFenn}, which will be described in 
Section~\ref{sec:surgery}.

A filling Dehn sphere, and its Johansson diagram, 
provides a presentation of the fundamental group of the filled 3-manifold. We 
describe this presentation in Section~\ref{sec:diagram-group}.
In Section~\ref{sec:complexity-4} we briefly analyze the case of filling 
Dehn spheres with at most $4$ triple points, proving the following theorem:

\begin{theorem}
  \label{thm:H1-complexity-4}
  If $\S$ is a filling Dehn sphere of $M$ with at most $4$ triple points, then 
  the first homology group $H_1(M,\Z)$ cannot be isomorphic to $\Z_3\oplus\Z_3$.
\end{theorem}

It is known that the lens space $L(3,1)$ has Montesinos complexity $2$
(cf.~\cite{Rthesis}). Thus, by Theorem \ref{thm:subadditivity-2},
$mc(L(3,1)\#L(3,1))\leq 6$.
The first homology group of $L(3,1)\#L(3,1)$ is known to be isomorphic to
$\Z_3\oplus\Z_3$, but
by Theorem \ref{thm:H1-complexity-4}, it cannot be $mc(L(3,1)\#L(3,1))=4$.
Therefore,

\begin{theorem}
  \label{thm:L31-L31-complexity-6}
  The Montesinos complexity of $L(3,1)\#L(3,1)$ is $6$.
\end{theorem}
 \begin{corollary}
  \label{cor:sharp-bound}
  The upper bound of Theorem~\ref{thm:subadditivity-2} is sharp.
\end{corollary}

We plan to classify all the fundamental groups of the manifolds with Montesinos
complexity up to 4 in a subsequent paper.

\section{Dehn spheres and their Johansson
diagrams}\label{sec:Dehn-spheres-Johansson-diagrams}

We will introduce some basic facts about Dehn spheres and their Johansson
diagrams. We refer to \cite{RHomotopies,Rthesis} for details.

In the following a \emph{curve} in the $2$-sphere $S^2$ or $M$  is the image of
an immersion
from $S^1$ into $S^2$ or $M$, respectively. A \emph{Dehn sphere} in $M$ is a
subset $\S\subset M$ such that there exists a transverse immersion
$f:S^2\rightarrow M$ such that $\S=f(S^2)$ (cf. \cite{Papa}). In this situation
we say that $f$ is a \emph{parametrization} of $\S$.

Let $\S$ be a Dehn sphere in $M$, and consider a parametrization $f$ of $\S$.

The \emph{singularities} of $\S$ are the points $x\in\S$ such that 
$\#f^{-1}(x)>1$. The \emph{singularity set} $S(\S)$ of $\S$ is the set of 
singularities of $\S$. As $f$ is transverse,  the singularities of $\S$ are
arranged along \emph{double curves}, and can be divided into \emph{double
points} ($\#f^{-1}(x)=2$), where two sheets of $\S$ 
intersect transversely, and \emph{triple points} ($\#f^{-1}(x)=3$), where three 
sheets of $\S$ intersect transversely.

Because $S^2$ is compact and without boundary, the double curves of $\S$ are
closed and 
there is a finite number of them. The triple points of $\S$ are isolated and
there is a finite number of them. Following \cite{Shima}, we denote by $T(\S)$
the set of triple points of $\S$.

The preimage under $f$ in $S^{2}$ of the singularity set of $\Sigma$, together
with the information about how its points become identified by $f$ in $\S$ is
the \emph{Johansson diagram} $\D$ of $\S$ (see~\cite{Montesinos}).

Because $S^2$ and $M$ are orientable, the preimage under $f$ of a double 
curve of $\S$ is the union of two different closed curves in $S^2$, and we will 
say that these two curves are \emph{sister curves} of $\D$. Thus, the Johansson 
diagram of $\S$ is composed by an even number of different closed curves in 
$S^2$. Indeed, we will  identify $\D$ with  the  set of different  curves
that 
compose it. For any curve $\aa\in\D$ we denote by $\ta$ the sister curve of 
$\aa$ in $\D$. This defines an involution $\tau:\D\rightarrow\D$ that sends 
each curve of $\D$ into its sister curve of $\D$.

The curves of $\D$ intersect with others or with themselves transversely at the
\emph{double 
points} of $\D$. The double points of $\D$ are the preimage under $f$ of 
the triple points of $\S$. If $P$ is a triple point of $\S$, the 
three double points of $\D$ in $f^{-1}(D)$ compose \emph{the triplet of} $P$ 
(see Figure \ref{fig:triple-point-00}).

\begin{figure}
\centering
\includegraphics[width=0.8\textwidth]{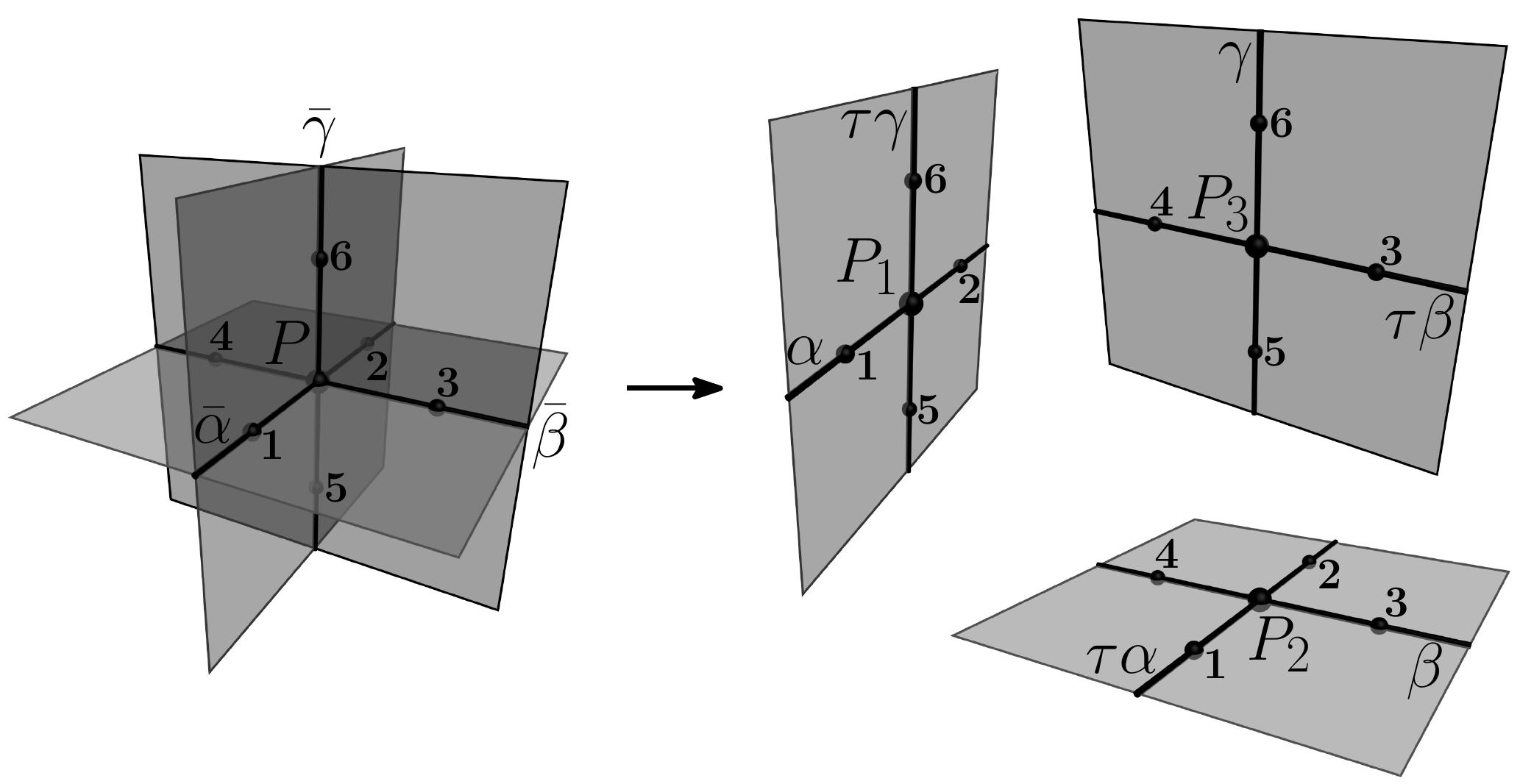}\\
\caption{A triple point of $\S$ and its triplet
in $S^2$}\label{fig:triple-point-00}
\end{figure}

The Dehn sphere $\S$ \emph{fills} $M$ if it defines a cell decomposition of $M$
whose $0$-skeleton is the set of triple points of $\S$, the $1$-skeleton is the
set of singularities of $\S$, and the $2$-skeleton is $\S$ itself. Equivalently,
$\S$ fills $M$ iff
\begin{enumerate}[\bf(F1)]
  \item $S(\S)-T(\S)$ is a disjoint union of open arcs;
  \item $\S-S(\S)$ is a disjoint union of open 2-disks;
  \item $M-\S$ is a disjoint union of open 3-balls.
  \label{F3}
\end{enumerate}
In particular, if $\S$ is filling each double curve must cross at least one
triple point and the 
Johansson diagram of $\S$ must be connected. A weaker version of filling Dehn 
spheres are the quasi-filling Dehn spheres, for 
which only condition~\textbf{(F3)} is required.

If we are given an \emph{abstract diagram}, i.e., an even collection of curves 
in $S^2$ coherently identified in pairs, it is possible to know if this 
abstract diagram is \emph{realizable}: if it is actually the Johansson diagram 
of a Dehn sphere in a $3$-manifold (see~\cite{Johansson1,Johansson2,Rthesis}). 
It is also possible to know if the abstract diagram is \emph{filling}: if it is 
the Johansson diagram of a filling Dehn sphere of a 3-manifold (see 
\cite{Rthesis}). If $\S$ fills $M$, it is possible to build $M$ out of the
Johansson diagram of $\S$. As every $3$-manifold has a filling Dehn sphere,
filling Johansson diagrams represent all closed, orientable $3$-manifolds.

In Figure~\ref{fig:diagrams-2-triple-points} we have depicted the simplest 
Johannson diagrams of filling Dehn spheres. In any case the curves must be 
identified in such a way that double points are identified with double points
and 
the arcs labelled with the same arrow become identified in the obvious way.
The graphs and the arrows give enough information about how all the points of
the diagram must become identified in $\S$. Nevertheless, for clarifying the
pictures we have labelled with the same name the double points that belong to
the same triplet.
The diagram of Figure~\ref{fig:shima_00-Johansson} is the classical diagram of 
I.~Johansson~\cite{Johansson1}. The $3$-sphere $S^3$ has only $3$ (up to
isotopy) filling Dehn spheres with $2$ triple points. They are part of the
A.~Shima's spheres given in \cite{Shima}. The corresponding Johansson diagrams
are those of Figures~\ref{fig:shima_00-Johansson},~\ref{fig:shima_01}
and~\ref{fig:shima_02}. 

The Johansson diagrams of Figures~\ref{fig:audi_01} 
and~\ref{fig:star-L31} appeared in~\cite{Rthesis}. They are, respectively, the 
Johansson diagram of a filling Dehn sphere of $S^2\times S^1$, and of $L(3,1)$.

It is well known that two closed curves in $S^2$ having nonempty transverse
intersection 
must have an even number of intersection points. Along the text we will refer to
this property as the \emph{even intersection property}.

\begin{figure}
\centering
\subfigure[]{\includegraphics[width=0.3\textwidth]{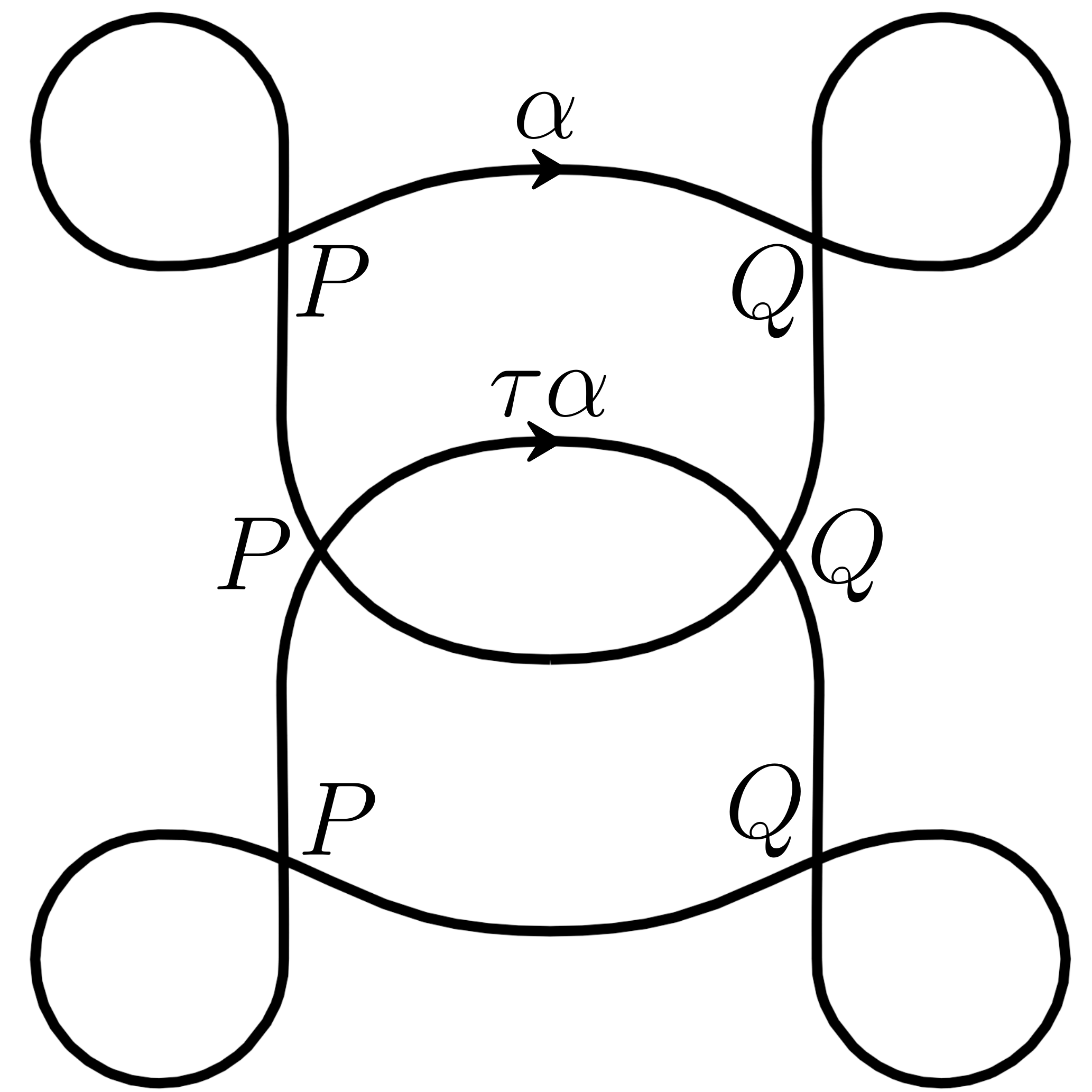}
\label{fig:shima_00-Johansson} }
\hfill
\subfigure[]{\includegraphics[width=0.3\textwidth]{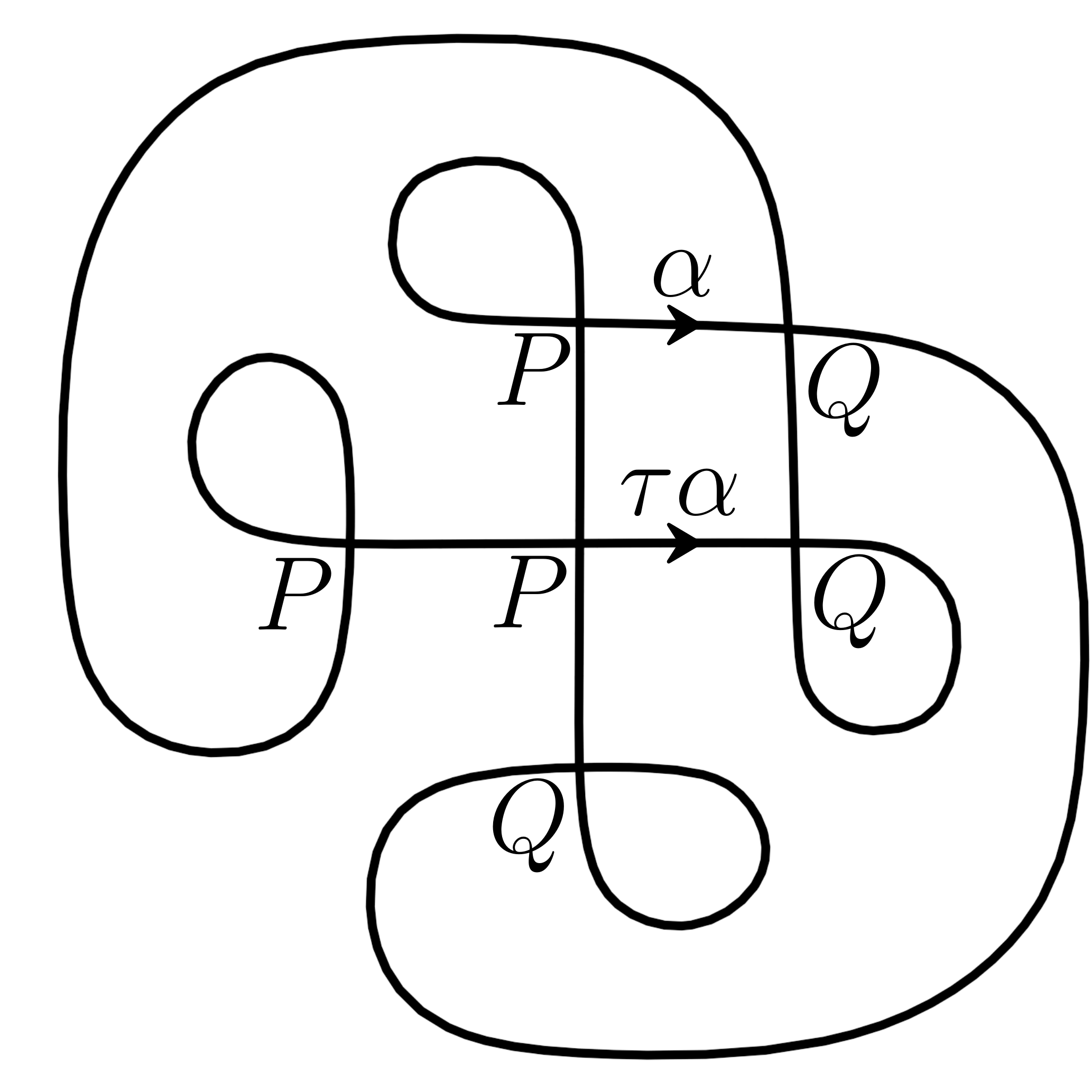}
\label{fig:shima_01} }
\hfill
\subfigure[]{\includegraphics[width=0.3\textwidth]{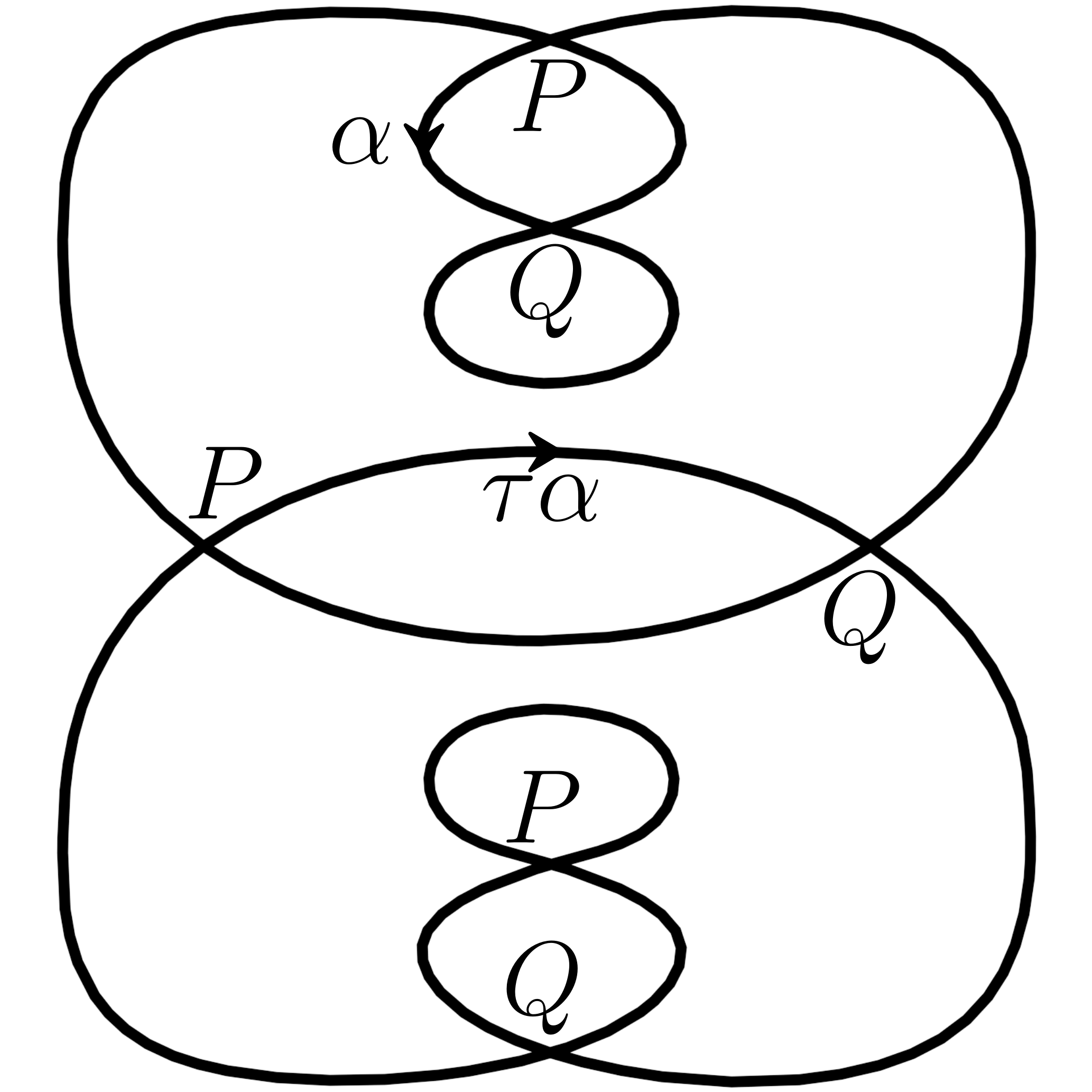}
\label{fig:shima_02} }\\
\hfill\subfigure[]{\includegraphics[width=0.3\textwidth]{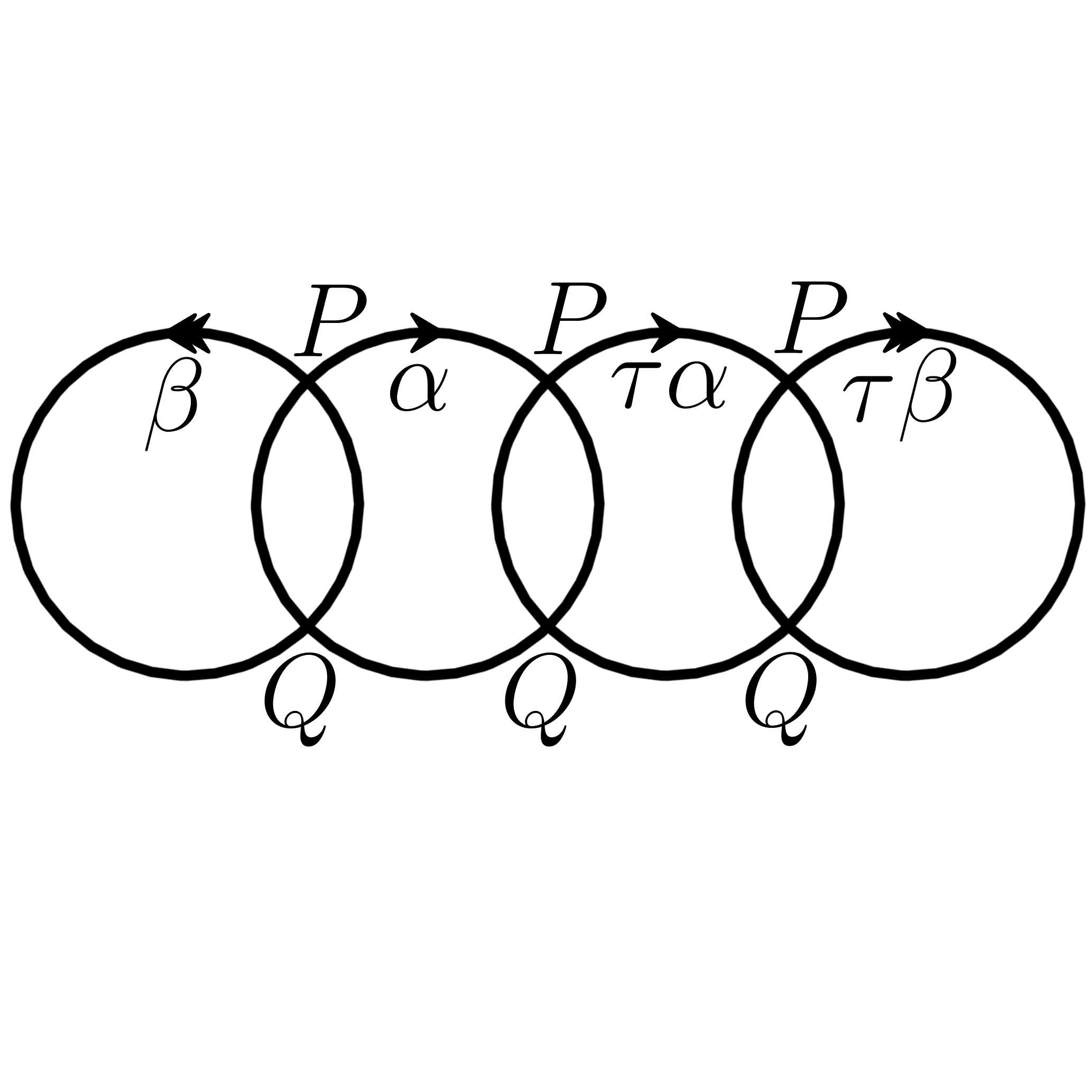}
\label{fig:audi_01} }
\hfill
\subfigure[]{\includegraphics[width=0.3\textwidth]{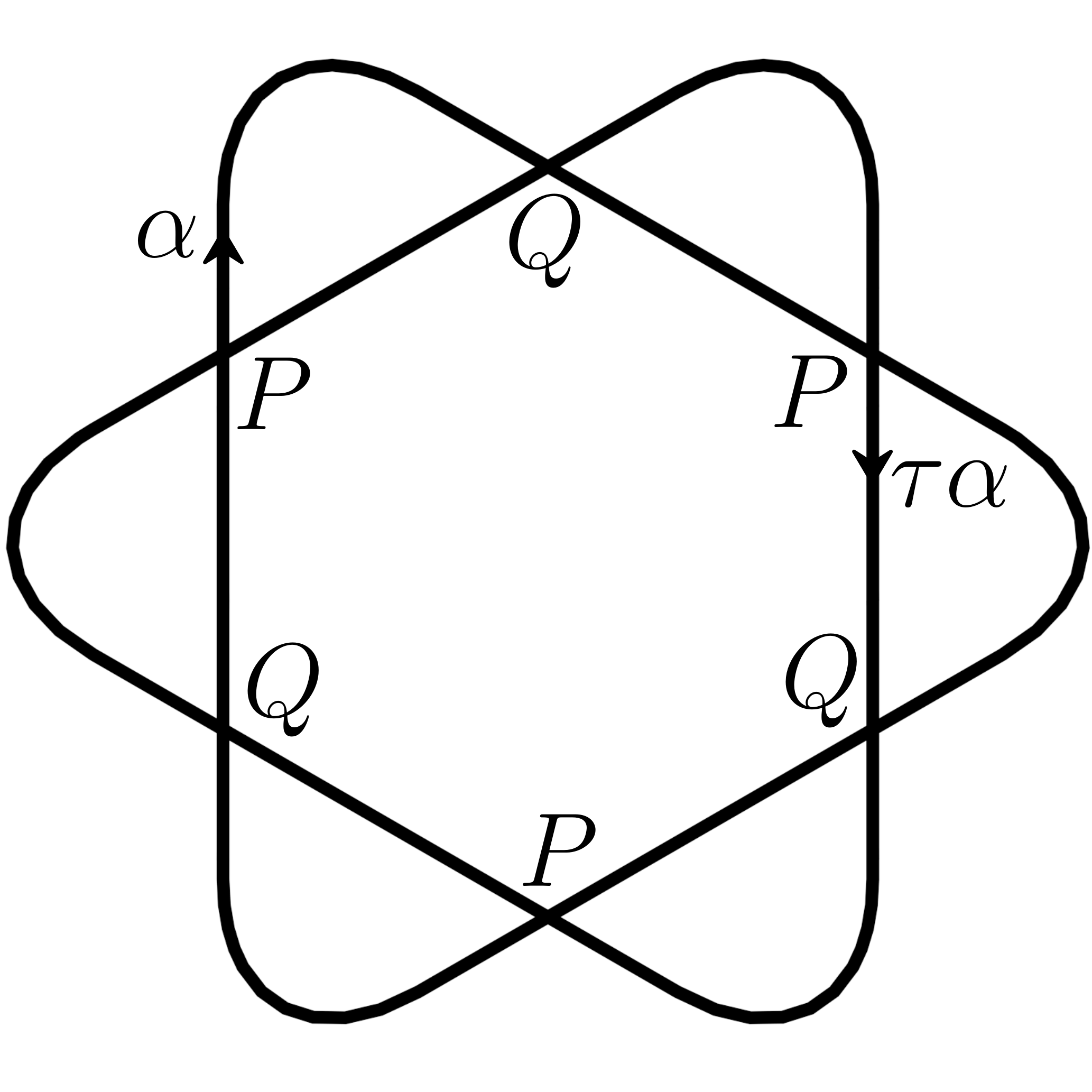}\label{fig:star-L31}
}\hfill\hfill

\caption{Filling Johansson diagrams with two triple
points}\label{fig:diagrams-2-triple-points}
\end{figure}

\begin{lemma}
If the Dehn sphere $\S$ has $p$ triple points and its Johannson diagram $\D$ is
connected, it can have at most $(2+3p)/4$ double curves.
\end{lemma}
\begin{proof}
We define an \emph{intersecting pair} of $\D$ as a pair of different curves of
$\D$ having nonempty intersection. Although $\D$ has $3p$ double points, by the
even intersection property it can have at most $3p/2$ distinct intersecting
pairs. As $\D$ is connected, it can have at most $1+3p/2$ different curves, and
so $\S$ can have at most $(1+3p/2)/2=(2+3p)/4$ double curves.
\end{proof}

In particular, a filling Dehn sphere with $2$ triple points can have at most 2
double curves, and a filling Dehn sphere with 4 triple points can have at most 3
double curves.

\section{Surgery on minimal Dehn spheres. Proof of Theorem
\ref{thm:subadditivity-2}}\label{sec:surgery}

Let $M_1$ and $M_2$ be two 3-manifolds, and let $\S_1$ and $\S_2$ be a filling
Dehn sphere
of $M_1$ and a filling Dehn sphere of $M_2$, respectively. Assume that $\S_1$
and $\S_2$ are minimal in $M_1$ and $M_2$, respectively. The
connected sum $M_1\#M_2$ is performed by removing the interior of two closed
3-balls $B_1$ and $B_2$ lying in $M_1$ and $M_2$ respectively. After that, in
the disjoint union of $M_1\setminus int(B_1)$ and $M_2\setminus int(B_2)$ 
the boundaries of $B_1$ and $B_2$ become identified by an homeomorphism.

\begin{figure}
\centering
\begin{tabular}{ll}
\subfigure[]{\includegraphics[width=0.3\textwidth]{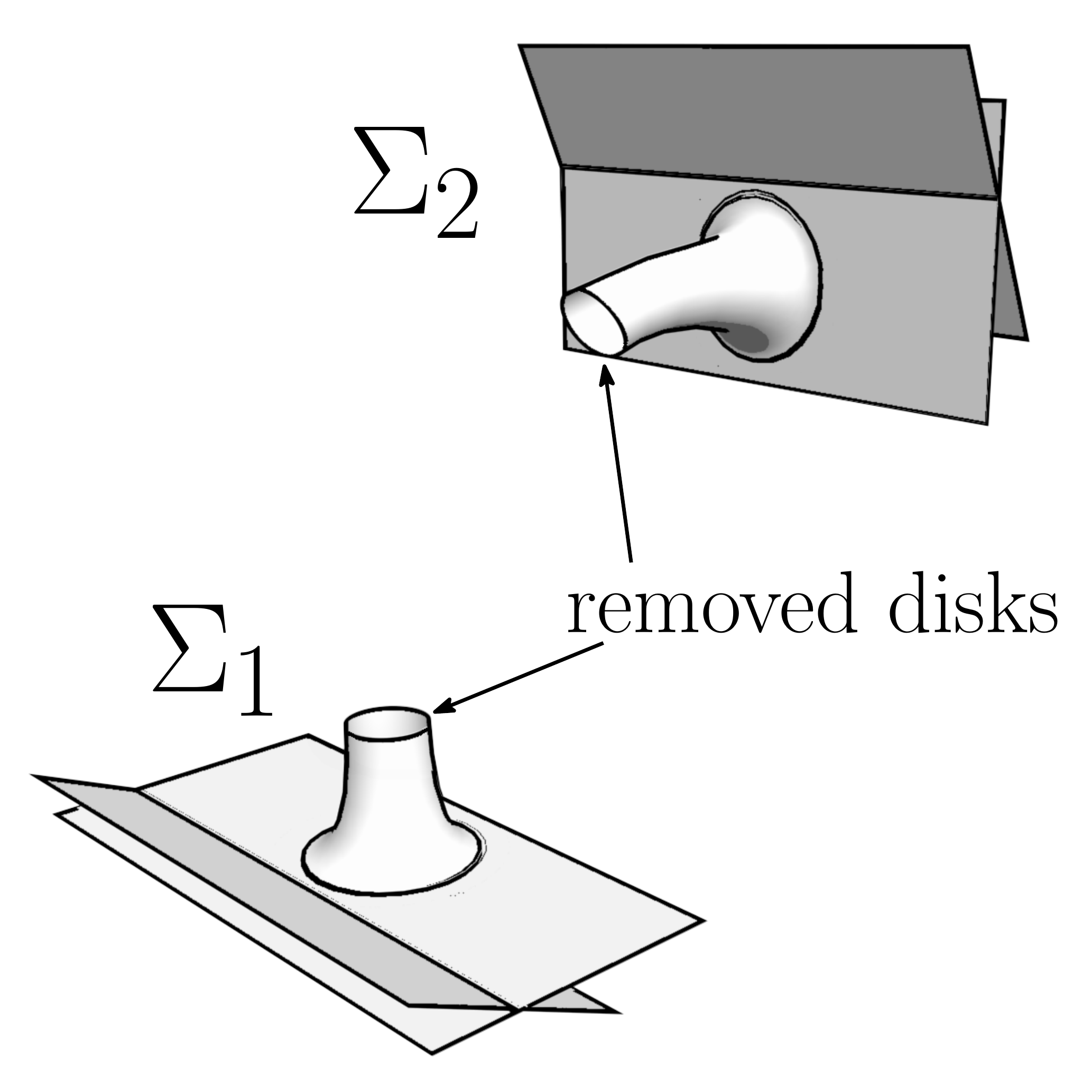}
\label{fig:surgery_00} }
&
\subfigure[]{\includegraphics[width=0.3\textwidth]{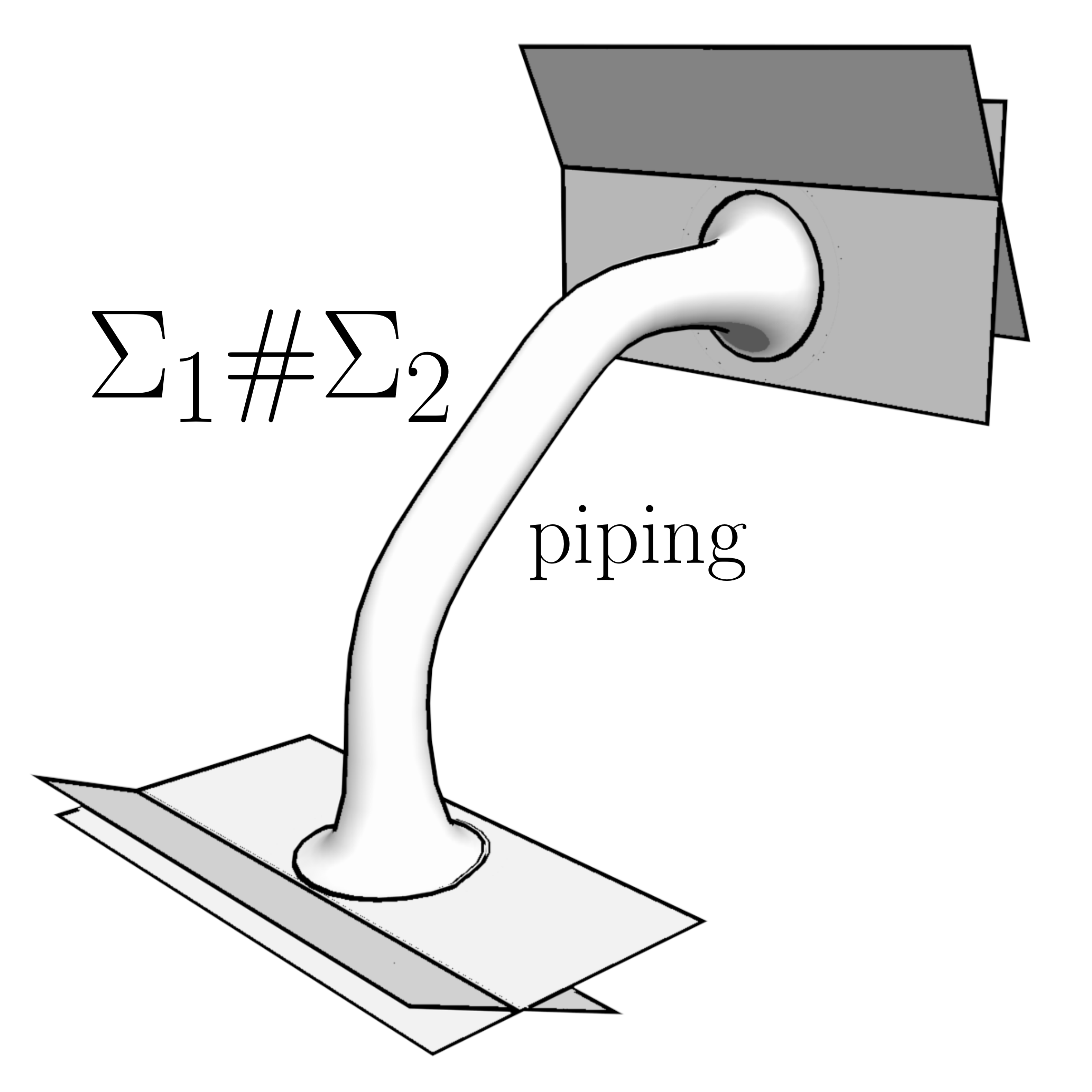}
\label{fig:surgery_01} }\\ 
\subfigure[]{\includegraphics[width=0.3\textwidth]{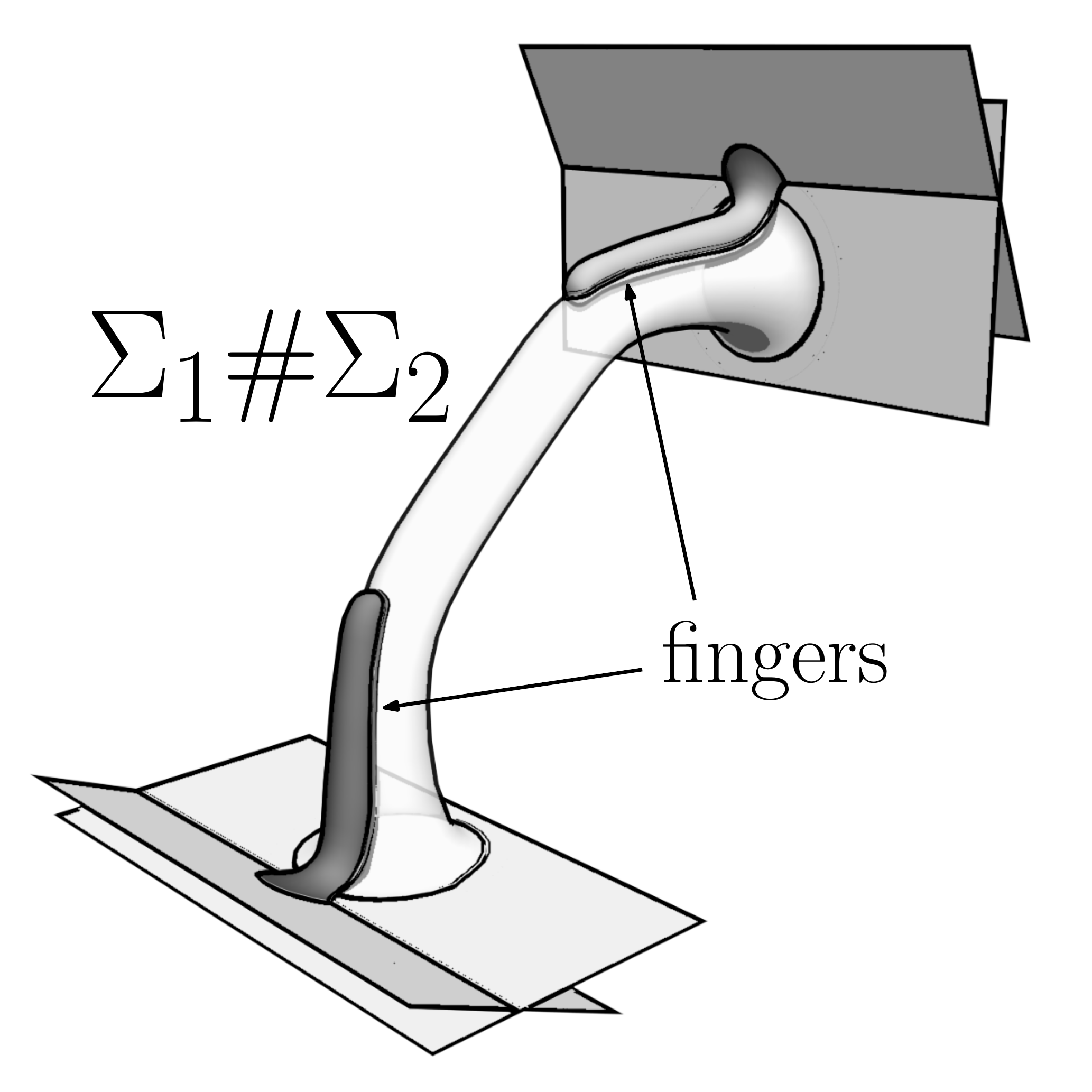}
\label{fig:surgery_02} }
&
\subfigure[]{\includegraphics[width=0.4\textwidth]{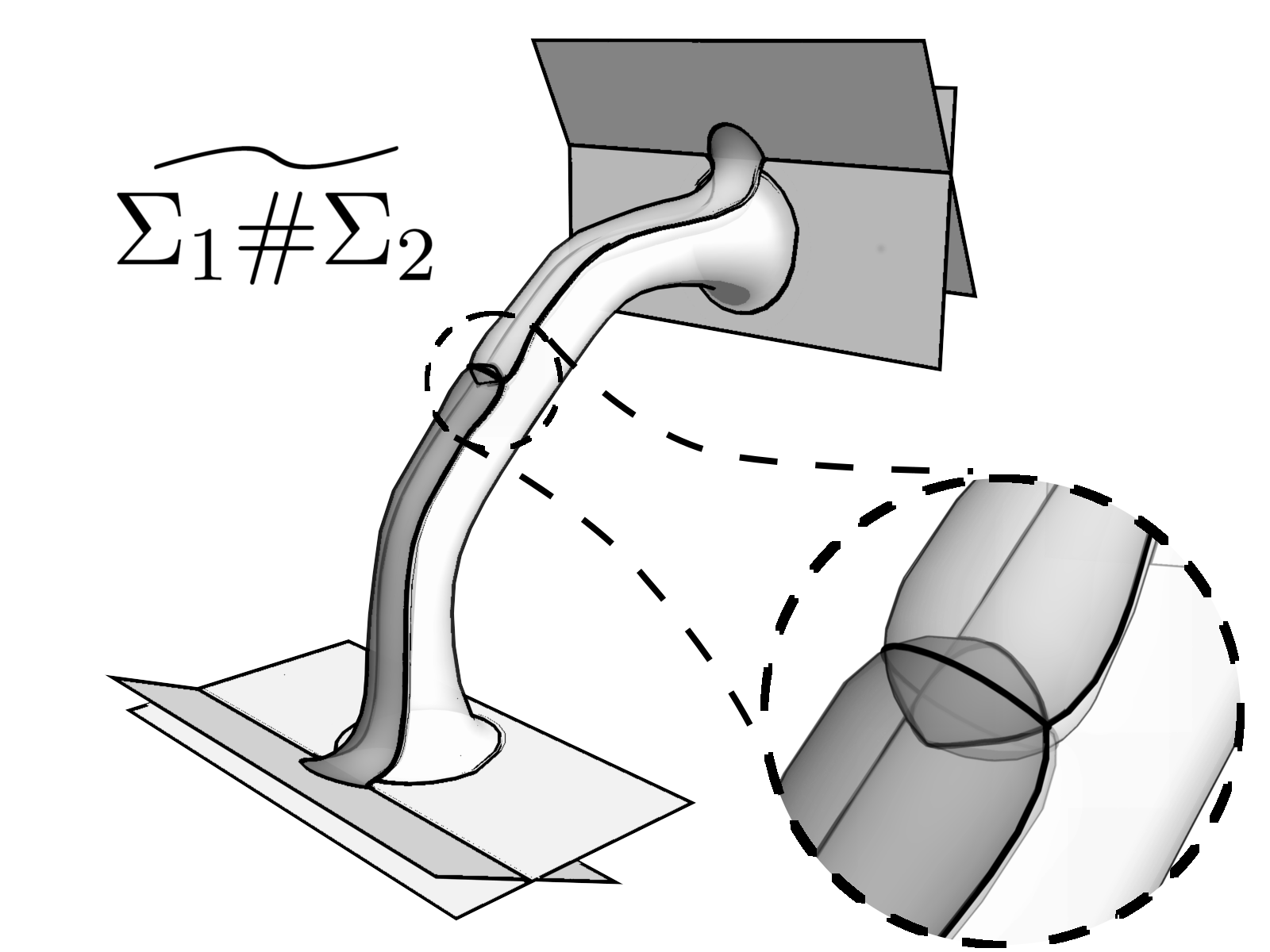}
\label{fig:surgery_03}}
\end{tabular}
\caption{Surgery between disjoint Dehn surfaces}\label{fig:surgery}
\end{figure}

If we choose the 3-balls $B_1$ and $B_2$ not intersecting $\S_1$ and $\S_2$,
respectively, the Dehn spheres $\S_1$ and $\S_2$ are transformed after the
connected sum into a pair of disjoint Dehn spheres of $\MM$, and the
connected component of $\MM\setminus(\S_1\cup\S_2)$ lying between
them is homeomorphic to $S^2\times I$, where $I$ is any open interval. We can
remove a small disk from $\S_1$ and from $\S_2$ (Figure \ref{fig:surgery_00}) in
order to connect them along a
\emph{piping} as in Figure \ref{fig:surgery_01}. After that, we obtain a Dehn
sphere
$\S_1\#\S_2$ which is not filling, but it is quasi-filling: the complementary
set of $\S_1\#\S_2$ in $\MM$ is a disjoint union of open 3-balls. The Dehn
sphere $\S_1\#\S_2$ is not filling because after the piping we have created a
connected component of $\S_1\#\S_2\setminus S(\S_1\#\S_2)$ which is
topologically an open annulus. This obstruction can be removed by throwing two
\emph{fingers}  (Figure \ref{fig:surgery_02}) along the piping between $\S_1$
and $\S_2$ until they intersect
as in Figure \ref{fig:surgery_03}, creating two new triple points. The resulting
Dehn sphere
$\widetilde{\S_1\#\S_2}$ is now a filling one, and 
it has $p_1+p_2+2$ triple points, where $p_1$ and $p_2$ are the number of triple
points of $\S_1$ and $\S_2$  respectively.

As $\S_1$ and $\S_2$ are minimal we have that $p_1=mc(M_1)$ and $p_2=mc(M_2)$.
This proves Theorem~\ref{thm:subadditivity-2}.

In Figure \ref{fig:surgery-diagram} we have illustrate the modifications to be
made on the Johansson diagrams of $\S_1$ and $\S_2$ in order to obtain the
Johansson
diagram of $\widetilde{\S_1\#\S_2}$ when $M_1$ and $M_2$ are two copies of
$L(3,1)$ an  $\S_1$ and $\S_2$ are two identical copies of the filling Dehn
sphere of $L(3,1)$ whose Johansson diagram is that of Figure \ref{fig:star-L31}.
The Johansson diagrams of $\S_1$ and $\S_2$ are two copies of the diagram of
Figure \ref{fig:star-L31} depicted in two different 2-spheres $S_1$ and $S_2$.
We have assumed that $M_1$
and $M_2$, $\S_1$ and $\S_2$, and $B_1$ and $B_2$ respectively are exact copies
of each other and that the homeomorphism that identifies $\partial B_1$ with
$\partial B_2$ is the identity map. With this assumptions, $\S_1$ and $\S_2$
become two specular copies of each other in $\MM$. If the disks removed from
$\S_1$ and $\S_2$ during the piping were also identical, we must paste two
specular copies of the same Johansson diagram as in Figure
\ref{fig:surgery-diagram} in order to obtain the Johansson diagram of
$\widetilde{\S_1\#\S_2}$ (Figure \ref{fig:surgery-diagram_03}).

\begin{figure}
\centering
\subfigure[]{\includegraphics[width=0.48\textwidth]{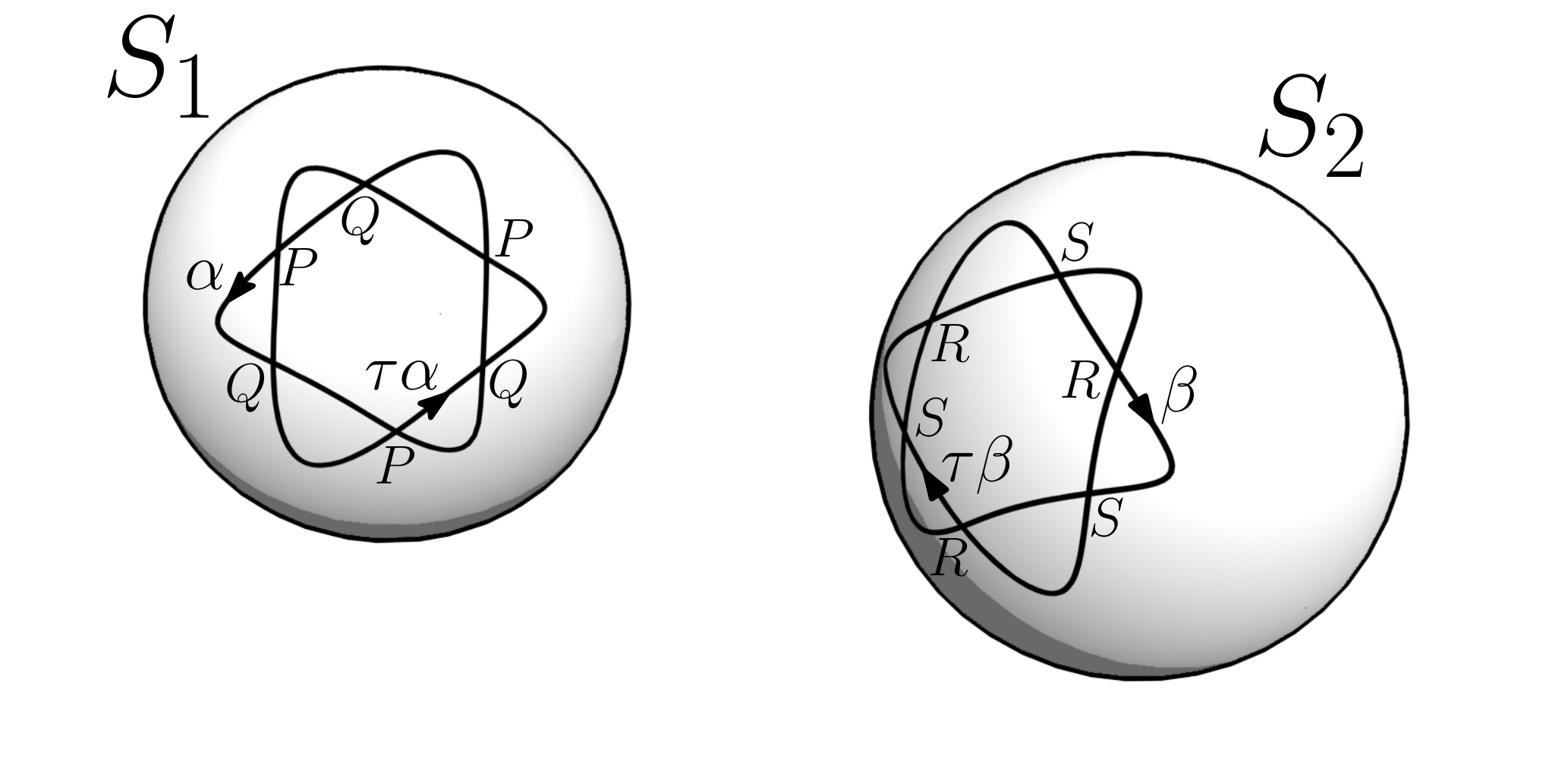}\label
{fig:surgery-diagram_00} }
\subfigure[]{\includegraphics[width=0.48\textwidth]{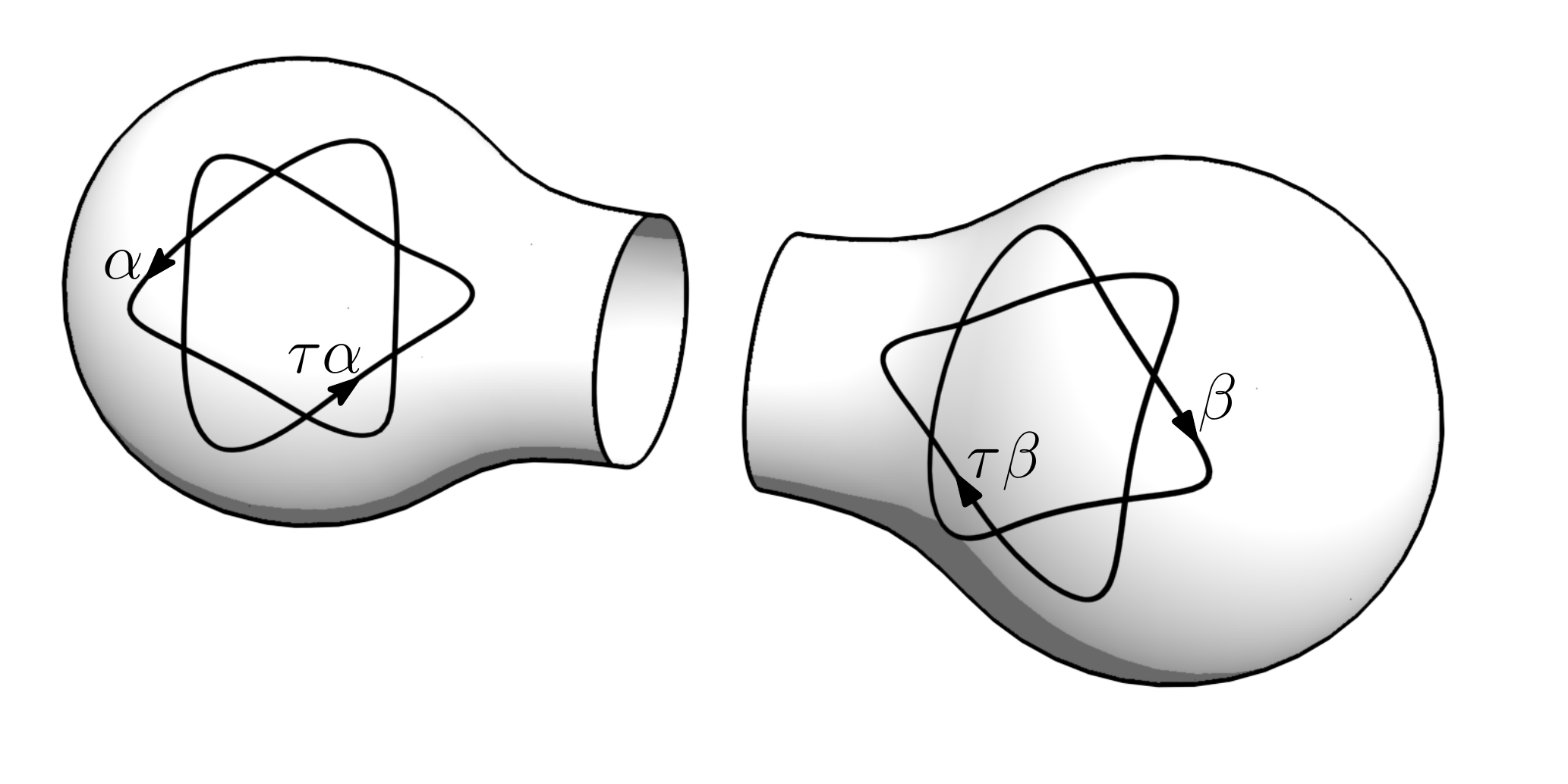}\label
{fig:surgery-diagram_01} }\\ 
\subfigure[]{\includegraphics[width=0.48\textwidth]{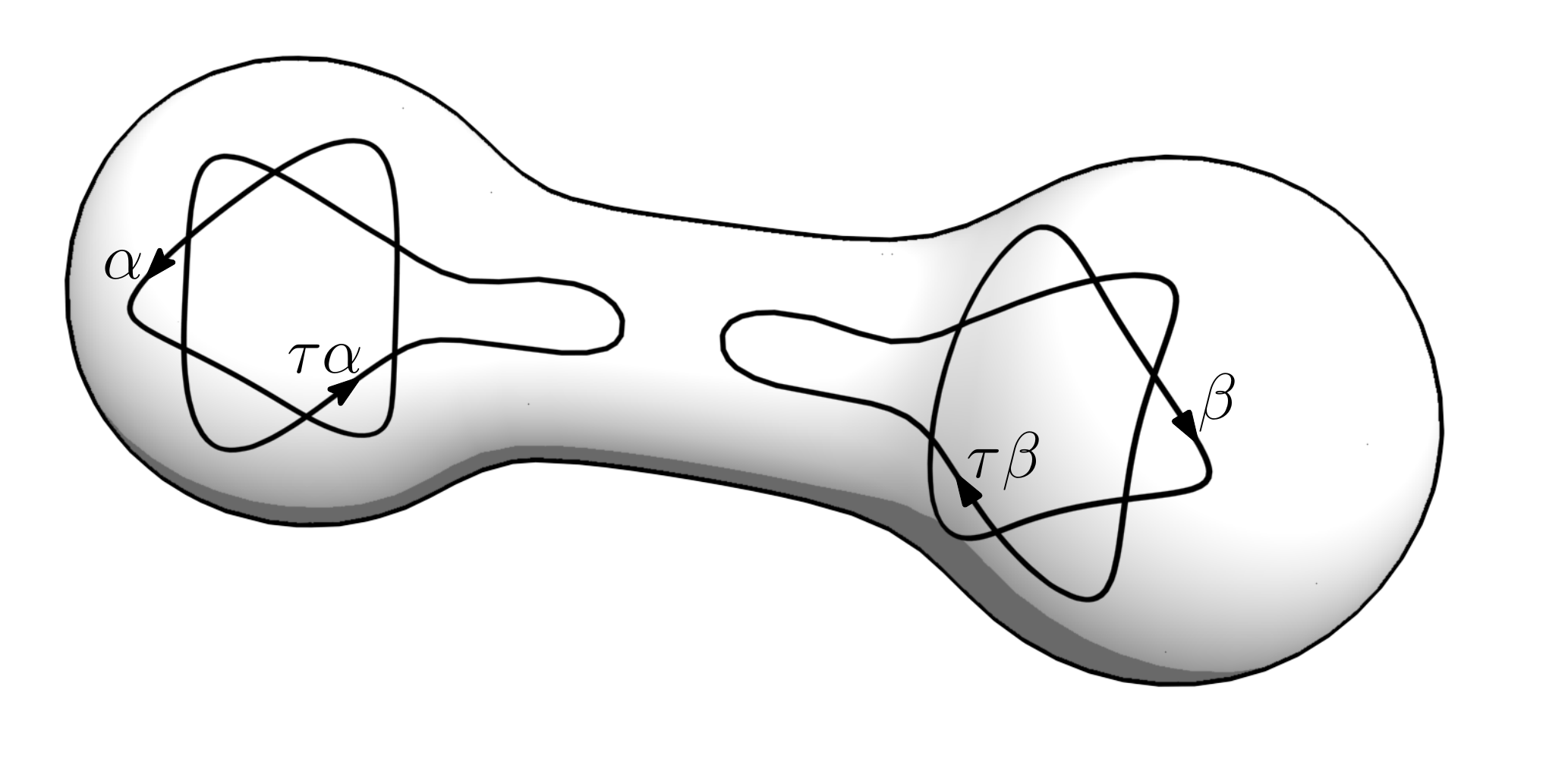}\label
{fig:surgery-diagram_02} }
\subfigure[]{\includegraphics[width=0.48\textwidth]{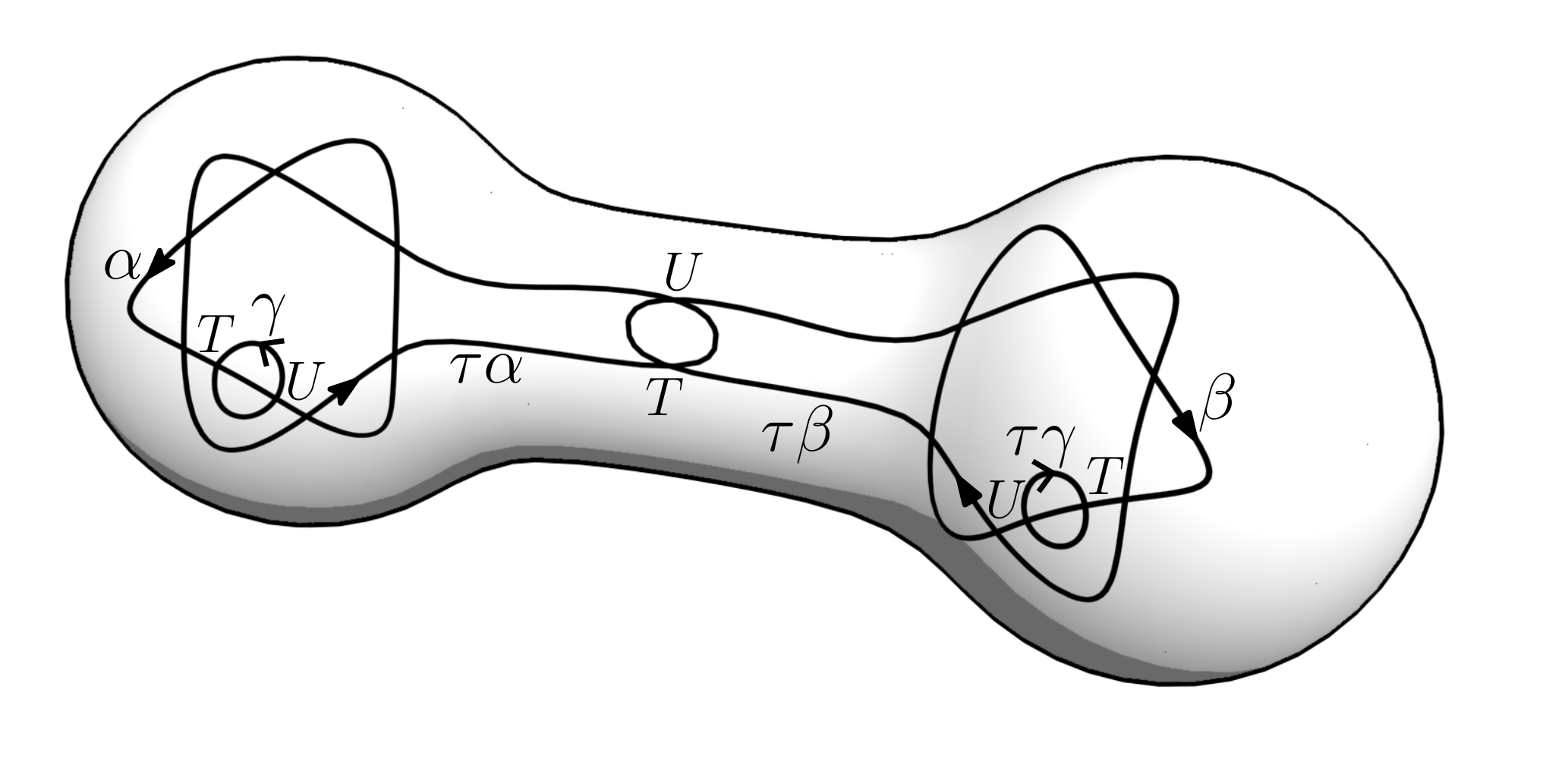}\label
{fig:surgery-diagram_03}}
\caption{Surgery between Johansson diagrams}\label{fig:surgery-diagram}
\end{figure}

\section{The diagram group}
\label{sec:diagram-group}

Let $\S$ be a filling Dehn sphere on $M$, let $f$ be a
parametrization of $\S$, and let $\D$ be the Johansson diagram of $\S$.

The Johansson diagram of $\S$ provides a presentation of the fundamental group 
of $M$ (see \cite{RHomotopies,Rthesis}). Let $\tau$ be the involution on the set
of curves of $\D$ that relates each curve with its sister curve. If
$\alpha_1,\alpha_2,\ldots,\alpha_{k}$ are the curves of $\D$, then $\pi_1(M)$ is
isomorphic to the \emph{diagram group}:
\begin{equation}
\pi(\D)=\langle\alpha_1,\ldots,\alpha_{k}|
\alpha_1\cdot\tau\alpha_1=\cdots=\alpha_{k}\cdot\tau\alpha_{k}=r_1=\ldots=r_p=1
\rangle\,,\label{eq:diagram-group}
\end{equation}
where the relators $r_i,i=1,\ldots,p$ are the \emph{triple point relators} of
$\D$
because they are given by the triple points of $\S$: if $P$ is a triple point of
$\S$ and we label its three  preimages in $S^2$ and the curves of $\D$
intersecting at them as in Figure~\ref{fig:triple-point-00}, the corresponding
relation is $r=\alpha\beta\gamma=1$. This presentation of $\pi_1(M)$ is due to
W. Haken (see Problem 3.98 of \cite{Kirby}). As $\aa_i$
and $\ta_i$, with $i=1,\ldots,k$, are inverse to each other in $\pD$, we will
use the notation $\aa_i^{-1}$ instead
of $\ta_i$ when we were
talking about elements of $\pD$.

At the triple point $P$ of $\S$ one, two or three different double curves of
$\S$ could intersect, and in each case we say that $P$ is a triple point of
\emph{type I},
\emph{type II} or \emph{type III}, respectively. We will
analyze these cases with more detail. Let $P_1,P_2,P_3$ be the triplet of $P$.

\paragraph{Type I triple points}
 
If the three arcs of double curve that intersect at $P$ belong to the same
double 
curve of $\S$, two things could happen:

\begin{itemize}
 \item \textbf{Type I.1}: one of the double points $P_1,P_2,P_3$ is a
self-intersection 
point of a curve of $\D$. If $P_1$, for example, is a self-intersecting point of
a curve $\aa$ of $\D$, then we are necessarily in the situation of
Figure~\ref{fig:triple-point-I1}: the other two double points of the triplet
must be an intersection point of $\aa$ with $\ta$ and a self-intersection point
of $\ta$. In this case the corresponding relation is $\aa\aa\aa^{-1}=1$, which
implies that $\aa=1$. We say also that the relation of $\pD$ obtained from $P$
is of \emph{type I.1}.
 \item \textbf{Type I.2}: none of the double points $P_1,P_2,P_3$ is a 
self-intersection point of a curve of $\D$. In this case, the three double
points $P_1,P_2,P_3$ are intersection points of a curve $\aa$ of $\D$ with its
sister curve $\ta$ (Figure~\ref{fig:triple-point-I2}). The corresponding
\emph{type I.2 relation} is $\aa^3=1$.
\end{itemize}

\paragraph{Type II triple points}

If two, but not three, of the three arcs of double curve that intersect at $P$
belong to 
the same double curve $\bar{\aa}$ of $\S$, and $\aa,\ta$ are the curves of $\D$
that project onto $\bar{\aa}$ under $f$, two possibilities arise:
\begin{itemize}
 \item \textbf{Type II.1\footnote{In this situation, in the notation of
\cite{Haken1}
it is said that the double
curve 
$f(\bb)$ of $\S$ is \emph{compensated}. A Dehn sphere such that each double
curve is compensated is simply
connected.}}: one of the three double points $P_1,P_2,P_3$ is a
self-intersec\-tion point of $\aa$ or $\ta$. If, for example, $P_1$ is a
self-intersection point $\aa$, the other two points $P_2,P_3$ must be
intersection points of $\ta$ with $\bb$ and $\tb$, where $\bb,\tb$ are curves of
$\D$ different from $\aa$ and $\ta$. We get a situation similar to that of
Figure~\ref{fig:triple-point-II1}, where the corresponding \emph{type II.1
relation} is $\aa\bb\aa^{-1}=1$, which is equivalent to
$\bb=1$.
 \item \textbf{Type II.2}: if none of the three double points $P_1,P_2,P_3$ is a
self-intersection point of $\aa$ or $\ta$, one of them, say $P_1$ must be a
intersection point of $\aa$ with $\ta$. Then, we have a configuration similar to
that of Figure~\ref{fig:triple-point-II2}, whose corresponding \emph{type II.2
relation} is $\aa\bb\aa=1$, which is equivalent to $\bb=\aa^{-2}$.
\end{itemize}

\paragraph{Type III triple points}

If the three arcs of double curve that intersect at $P$ belong to different
double 
curves, then we can label the curves of $\D$ to obtain a configuration as that
of Figures~\ref{fig:triple-point-00} and~\ref{fig:triple-point-III}, whose
\emph{type III relation} is $\aa\bb\cc=1$.

$\D$) with a vertex $[\aa]$ represented by each curve of $\aa\in\D$ and the
edges given by the double points of $\D$. If $\aa$ and $\bb$ intersect each
other at $m$ different double points of $\D$, the graph $G_\D$ will have exactly
$m$ edges joining $[\aa]$ and $[\bb]$, and if the curve $\aa$ has $n$
self-intersection points, the graph $G_\D$ has $n$ edges joining $[\aa]$ with
itself. With these assumptions, because the number of intersection points
between two different curves is even, each vertex of $G_\D$ has even degree.
Though the number of edges of $G_\D$ is $3p$, two different closed curves in
$S^2$ with transverse intersection intersect each other at an even number of
double points, and so the number of pairs of different vertices of $G_\D$ which
are adjacent is at most $3p/2$.
%
so $G_\D$ can have at most $1+3p/2$ vertices. This gives an upper bound to the
number of double curves of $\S$ in terms of the number of triple points of $\S$
for a Dehn sphere $\S$. Note that the unique property of $\S$ related to
fillingness that we have used is the connectedness of the Johansson diagram of
$\S$. Thus, we have:
whose Johannson diagram is connected can have at most $(2+3p)/4$ double curves.	
double curves, and a filling Dehn sphere with 4 triple points can have at most 3
double curves.
We will denote by $1$ the trivial group and by $\Z_q$, with $q=2,3,\ldots$, the
cyclic group with $q$ elements. For any two groups $H,G$, we write
$H\cong G$ when both groups are isomorphic, and $H\lesssim G$ when $H$ is
isomorphic to a subgroup of $G$.

\begin{theorem}\label{thm:at-most-two-curves}
If $\S$ is a filling Dehn sphere of $M$ with at most two double curves, then
if $\pi_1(M)$ is not trivial it is isomorphic to
$\Z,\Z_2,\Z_3,\Z_4,\Z_5$ or $\Z_6$.
\end{theorem}
\begin{proof}
If $\S$ has only one double curve, there is one generator of the $\pD$, and all
the triple points are of type I. It is $\pD=1$ or $\pD\cong \Z_3$.

Assume now that $\S$ has two double curves and let $\aa,\ta,\bb,\tb$ be the
curves of the Johansson diagram $\D$ of $\S$. Because $\S$ is filling, the
singularity set of $\S$ and the diagram $\D$ are connected, and so there must be
at least one type II triple point $P$ in $\S$.

If there is one triple point of type II.1, we can assume that
of Figure~\ref{fig:triple-point-II1}, and so 
the relation $\bb=1$ holds for the diagram group $\pD$. Consequently, $\pD$ is
the cyclic group generated by $\aa$. If there is a type II relation not
equivalent to $\bb=1$:
\begin{itemize}
 \item the relators $\aa\bb\aa$ or $\aa\bb^{-1}\aa$ would imply $\aa^2=1$, and
so it is $\pD\lesssim\Z_2$;
 \item $\bb\aa\bb$, $\bb\aa^{-1}\bb$, $\bb\aa\bb^{-1}$ or $\bb\aa^{-1}\bb^{-1}$
lead to $\aa=1$, so $\pD$ is trivial.
\end{itemize}
Assume that all the type II relations in $\pD$ are equivalent to $\bb=1$. If
there's 
no type I relation involving the generators $\aa$ or $\ta=\aa^{-1}$, the
generator $\aa$ 
would be free and so it is $\pD\cong\Z$. If there are type I relators
involving 
$\aa$ or $\aa^{-1}$, we would have $\pD\cong 1$ or $\pD\cong \Z_3$. 

If there is no triple point of type II.1 we can assume, by renaming the curves
of $\D$ if 
necessary, that the relation $\aa\bb\aa=1\Leftrightarrow
\bb=\aa^{-2}$ holds in $\pD$. Again, $\pD$ is the cyclic group generated by
$\aa$.
either $\Z$ or 
If there is another relation in $\pD$ not equivalent to 
$\bb=\aa^{-2}$:
\begin{itemize}
 \item $\aa\bb^{-1}\aa=1$ leads to $\aa^4=1$, and so $\pD\lesssim\Z_4$;
\item $\bb\aa\bb=1$ would imply that $\aa^3=1$ and so $\pD\lesssim\Z_3$;
\item $\bb\aa^{-1}\bb=1$ gives $\aa^5=1$ and so $\pD\lesssim\Z_5$;
\item $\aa=1$ makes $\pD$ trivial;
\item $\aa^3=1$ implies that $\pD\lesssim\Z_3$;
\item $\bb=1$ makes $\pD\lesssim\Z_2$;
\item $\bb^3=1$ leads to $\aa^6=1$, and so $\pD\lesssim\Z_6$.
\end{itemize}
\end{proof}

\begin{figure}
\centering
\subfigure[$\aa\aa\aa^{-1}=1\iff
\aa=1$]{\includegraphics[width=0.8\textwidth]{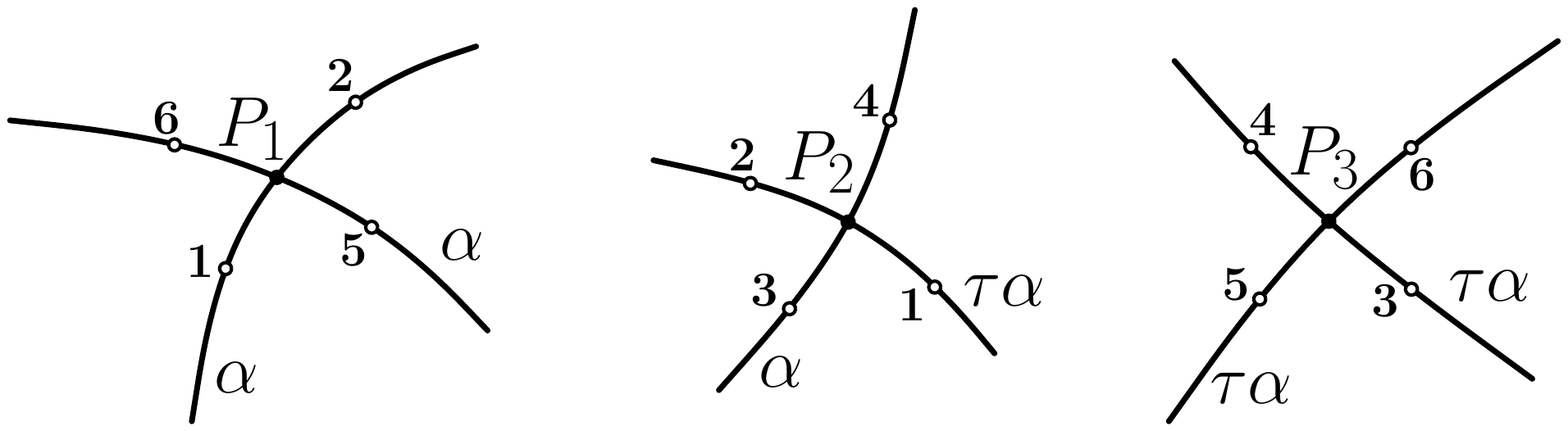}
\label{fig:triple-point-I1} }\\
\subfigure[$\aa^3=1$]{\includegraphics[width=0.8\textwidth]{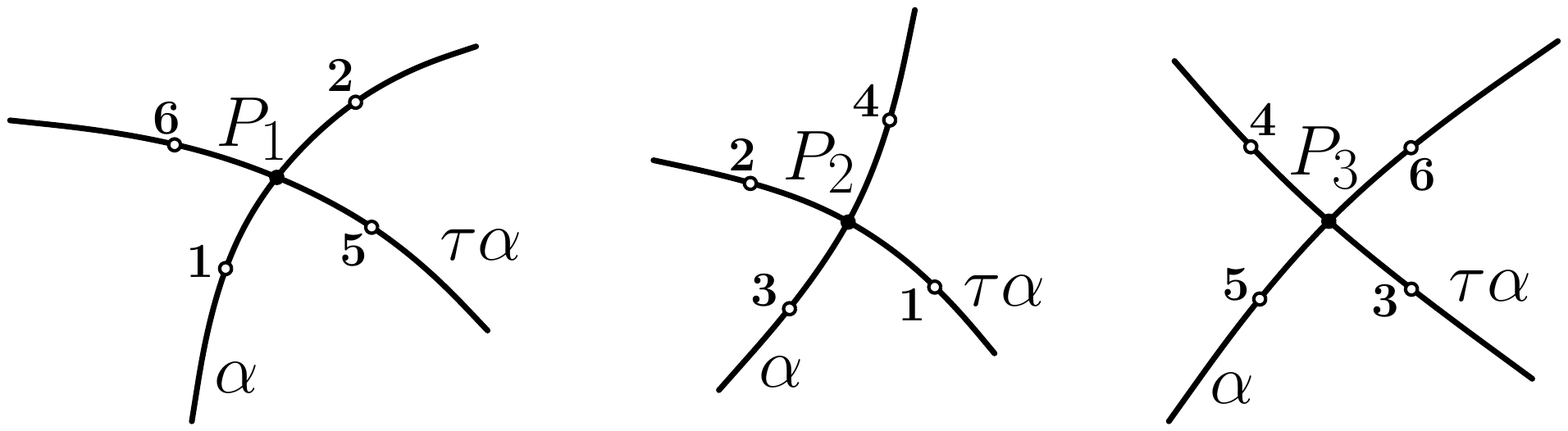}
\label{fig:triple-point-I2} }\\
\subfigure[$\aa\bb\aa^{-1}=1\iff\bb=1$]{\includegraphics[width=0.8\textwidth]
{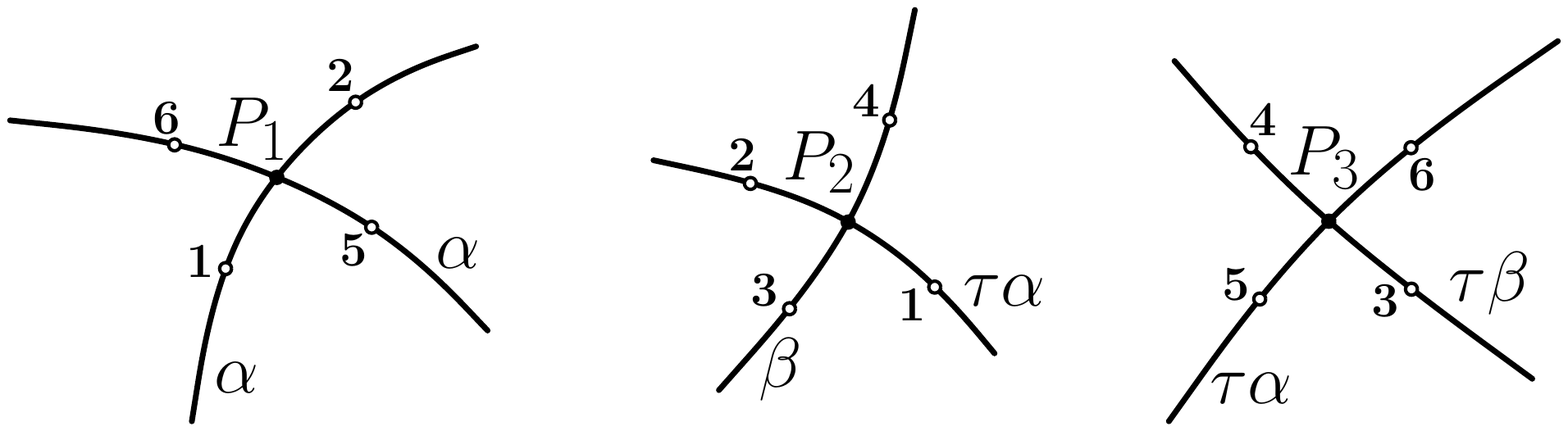}\label{fig:triple-point-II1} }\\
\subfigure[$\aa\bb\aa=1\iff
\bb=\aa^{-2}$]{\includegraphics[width=0.8\textwidth]{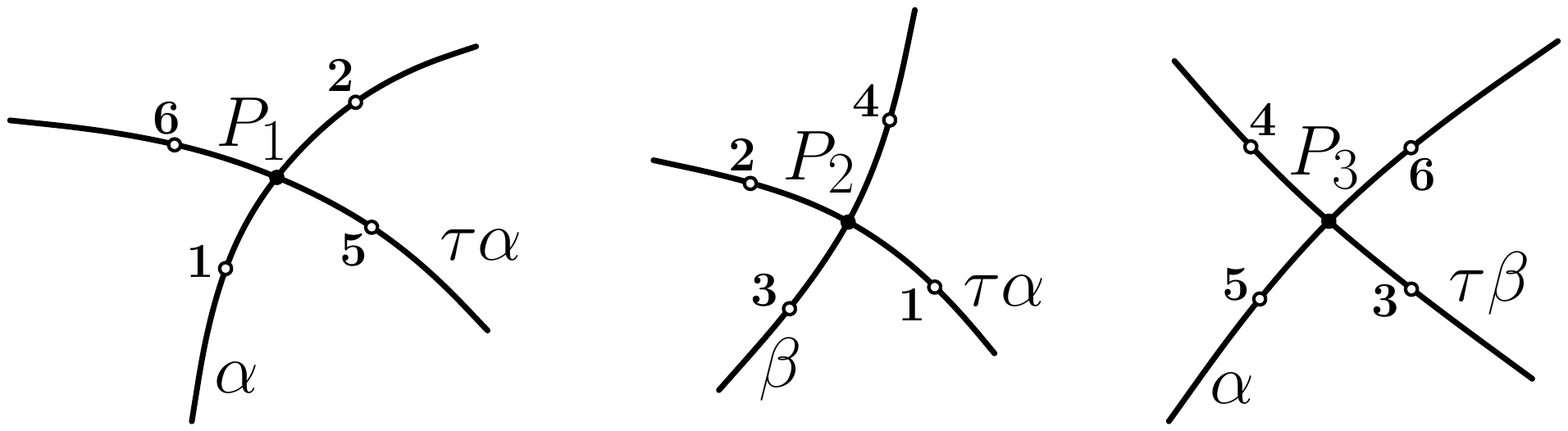}
\label{fig:triple-point-II2} }\\
\subfigure[$\aa\bb\cc=1$]{\includegraphics[width=0.8\textwidth]{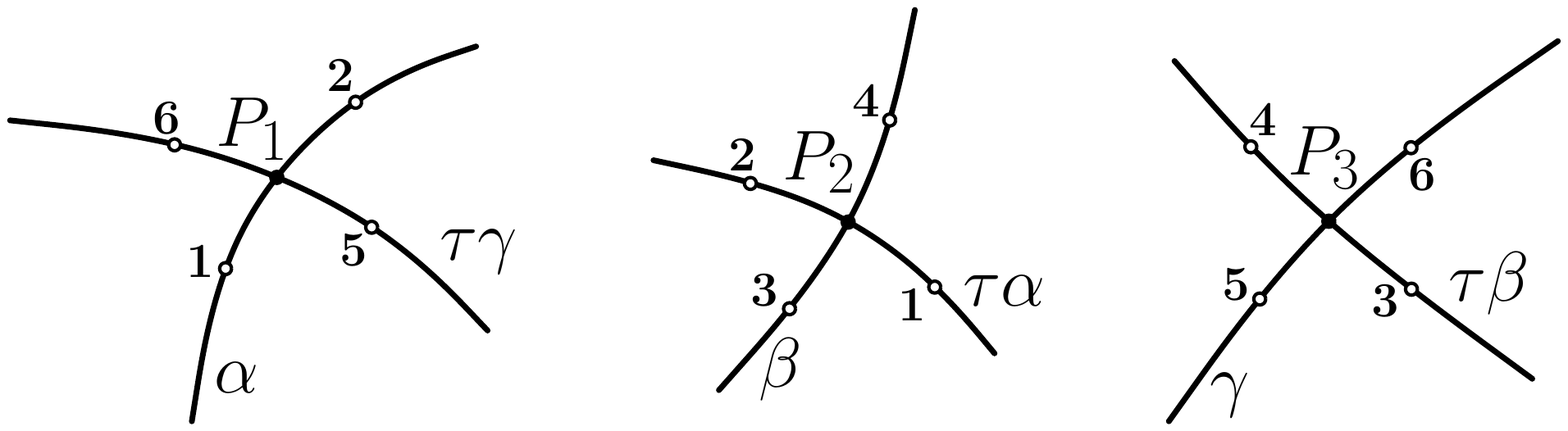}
\label{fig:triple-point-III} }\\
\caption{the curves of $\D$ around a triplet of double
points}\label{fig:triple-points}
\end{figure}

\section{3-manifolds with Montesinos complexity 4. Proof of Theorem
\ref{thm:H1-complexity-4}}\label{sec:complexity-4}

Let $\S$ be a filling Dehn sphere on $M$, $f$ a
parametrization of $\S$, and $\D$ the Johansson diagram of $\S$.

\begin{lemma}\label{lem:no-type-I-triple-point}
 If $\S$ has three double curves and four triple points, it has no type I triple
point.
\end{lemma}

\begin{proof}
Let $P$ be a type I triple point of $\S$, and let $Q,R,S$ be the other three
triple points 
of $\S$. Let $\bar{\aa}$ be the double curve of $\S$ through $P$, and
let $\aa,\ta$ be the curves of $\D$ that are projected onto $\bar{\aa}$ under
$f$. The triple point $P$ can be of type I.1 or of type I.2, but in both
cases there is an odd number of intersection points 
between $\aa$ and $\ta$ in the triplet of $P$ (see
Figures~\ref{fig:triple-point-I1} and~\ref{fig:triple-point-I2}). By the even
intersection property, there must be another 
intersection point $Q_1$ of $\aa$ with $\ta$ out of the triplet of $P$. We can
assume that $Q_1$ belongs to the triplet of $Q$.

The unique types 
of triple points where a curve of $\D$ intersects its sister curve are types I
and II.2 (see 
Figure~\ref{fig:triple-points}).
If $Q$ is a type II.2 triple point, after renaming the curves $\bb,\tb,\cc,\tc$
if necessary, we can assume also that the curves $\aa,\ta,\bb,\tb$ intersect at
the triplet of $Q$ as in Figure~\ref{fig:triple-point-II2}. As $\D$ is
connected, by the even intersection property, among the remaining six double
points of $\D$ lying in the triplets of $R$ and $S$ there must be at least:
\begin{enumerate}
 \item another intersection point of $\ta$ with $\bb$;
 \item another intersection point of $\aa$ and $\tb$;
 \item renaming $\cc$ and $\tc$ if necessary, two intersection points of $\cc$
with one of the curves $\aa,\ta,\bb,\tb$; and
 \item two intersection points of $\tc$ with one of the curves
$\aa,\ta,\bb,\tb,\cc$.
\end{enumerate}
It is not difficult to check that with this restrictions, each of the remaining
two triplets of 
$\D$ must involve the six curves of $\D$. If the triplet of $R$ contains a point
of $\ta\cap\bb$, for example, after renaming $\cc,\tc$ if necessary we can
assume that the intersection of the curves of $\D$ around $f^{-1}(R)$ is as that
of Figure~\ref{fig:triple-point-III}. Then, $f^{-1}(S)$ must contain: an
intersection point of $\aa$ and $\tb$, an intersection point of $\gamma$ and
$\tb$ and an intersection point of $\aa$ and $\tc$, which is impossible because
there fail to appear $\ta$ and $\bb$. This means that $Q$ cannot be a type II.2
triple point.

Therefore, $Q$ is a type I triple point. By the even intersection property, for
creating a connected diagram with $\aa\cup\ta$ and $\bb,\tb,\cc,\tc$ we need to
introduce at least eight double points, but there are only six remaining double
points in $\D$. This leads to a contradiction, and so there cannot be a triple
point of type I.
\end{proof}

\begin{proof}[Proof of Theorem \ref{thm:H1-complexity-4}]
As $\pD$ is isomorphic to $\pi_ 1(M)$, the abelianized $\aD$ of $\pD$ is
isomorphic to $H_1(M,\Z)$. We will work with $\aD$, and we will show that it
cannot be isomorphic to $\Z_3\oplus\Z_3$, the abelianized group of
$\Z_3\ast\Z_3$.
We will use the same names for the generators
of $\aD$ and $\pD$, and we will give the same names (type I.1, I.2, II.1, II.2
and III) to the abelianized relations in $\aD$ as their original relations in
$\pD$.
By Theorem~\ref{thm:at-most-two-curves}, we can restrict our analysis to
the case when $\S$ has three double curves, and by
Lemma~\ref{lem:no-type-I-triple-point}, in this case there is no type I triple
point in $\S$. Therefore all the relations in $\pD$ are of type II or III.

If $\pD$ has a type III relation, we can assume that in $\aD$ holds the relation
\begin{equation}\label{eq:abelianized-type-III}
\aa+\bb+\cc=0\,. 
\end{equation}

If all the relations of $\aD$ are equivalent to \eqref{eq:abelianized-type-III},
then $\aD$ is isomorphic to $\Z\oplus\Z$.

If $\aD$ has another type III relation not equivalent to
\eqref{eq:abelianized-type-III}, after renaming the curves of $\D$ we can assume
that this relation is
$$\aa-\bb-\cc=0\,.$$
This relation, together with \eqref{eq:abelianized-type-III} gives $2\aa=0$. Is
$\aa$ is trivial in $\aD$, then by \eqref{eq:abelianized-type-III} it is
$\bb=-\cc$ and $\aD$ is cyclic. If $\aa$ is not trivial, $\aD$ has an element of
order two, and so $\aD$ cannot be isomorphic to $\Z_3\oplus\Z_3$.

If $\aD$ has a type II relation, we
can assume that it is $\bb=0$ or $\bb=-2\aa$. In any case, this relation,
together with \eqref{eq:abelianized-type-III}, implies that $\aD$ is cyclic.

Thus, if $\aD$ has a type III relation, it cannot be isomorphic to 
$\Z_3\oplus\Z_3$.

\medskip
Assume now that all the relators are of type II.
If there is a type II.1 relation, we can assume that the relation $\bb=0$
holds in $\aD$ (see Figure~\ref{fig:triple-point-II1}). 
We have that:
\begin{itemize}
 \item If the remaining three relations are equivalent to $\bb=0$, $\aD$ is a
free abelian group of rank two.
 \item If there is a relation of the type $\aa=0$, $\aa=\pm 2\cc$, $\aa=\pm
2\bb$, perhaps interchanging $\aa$ with $\cc$, the group $\aD$ is cyclic.
 \item If there is a relation of the type $\bb=\pm 2\aa$, then
$\aD$ is cyclic (if $a=0$) or it has elements of order two (if $a\neq 0$). The
same holds if we have $\bb=\pm 2\cc$.
\end{itemize}

If the four relators are of type II.2, we can assume that one of them gives the
relation $\bb=2\aa$. Then
\begin{itemize}
 \item If the remaining three relations are equivalent to $\bb=2\aa$, then $\aD$
is free abelian of rank two.
 \item If there is a relation of the type $\aa=\pm 2\cc$, the group $\aD$ is
cyclic.
 \item If there is a relation of the type $\cc=\pm 2\aa$, $\aD$ is cyclic.
 \item If $\bb=\pm 2\cc$ holds, we have that $2\aa=\pm 2\cc$. If it is
$2\aa=2\cc$,
by taking $\aa,\aa-\cc$ as generators of $\aD$ we have that $\aD$ must be cyclic
(if $\aa-\cc=0$) or it must contain elements of order two (if $\aa-\cc\neq 0$).
The same argument can be applied when the relation $2\aa=-2\cc$ holds in $\aD$.
 \item If there is a relation of the type $\cc=\pm2\bb$, $\aD$ is cyclic.
 \item If there's no type II.2 relation involving $\cc$, the generator $\cc$
of $\aD$ is free and so $\aD$ has rank at least 1.
\end{itemize}

The proof is complete.
\end{proof}

\section{Comments}\label{sec:comments}

With a bit more effort, we can extend the techniques of the proofs of Theorems
\ref{thm:at-most-two-curves} and \ref{thm:H1-complexity-4} for obtaining a
list of candidates for fundamental groups of manifolds with Montesinos
complexity up
to 4. This will be made in a subsequent paper, where the following
theorem is proved:
\begin{theorem}\label{thm:fund-groups-complexity-4}
If the $3$-manifold $M$ has Montesinos complexity $mc(M)\leq 4$, the
fundamental group of $M$, if it is not trivial, is isomorphic to either
$\Z$, $\Z_q$ with $q\leq 6$, $\Z\oplus\Z$ or to the groups:
\begin{align*}
	G_1=&\langle a,b|ab^{-1}=ba\rangle\,,
	\\G_2=&\langle a,b|a^2=b^2\rangle\,.
\end{align*}

\end{theorem}
The proof of this theorem relies in a combinatorial study of the groups having
at most three generators and four triple point relations. The combinatorial
properties of the filling Johansson diagrams: (i) connectedness; (ii) even
intersection
property; and (iii) the \emph{symmetry} between sister curves (when performing a
complete travel along sister curves we must cross the same number of double
points); impose strong combinatorial restrictions on the diagram
groups. For an arbitrary group
$\mathscr{G}$, we can wonder if there exists a
\emph{Haken presentation} of $\mathscr{G}$: a presentation similar to those of
the diagram groups (generators and triple point relations) of filling Johansson
diagrams with the same combinatorial restrictions as those imposed by properties
(i), (ii) and (iii) above. For a
given group $\mathscr{G}$ having a Haken presentation we can define its
\emph{Haken complexity} $hc(\mathscr{G})$ as the minimal number of triple point
relators among all its Haken presentations. Of
course, the fundamental group of a $3$-manifold $M$ has a Haken presentation
and it is always $hc(\pi_1(M))\leq mc(M)$. The
question: \textit{is it always $hc(\pi_1(M))= mc(M)$?}, naturally
arises. A possitive answer to this question would be highly nontrivial to prove
because, in particular, it would imply a solution of the Poincar\'e Conjecture.
We don't know
if all the fundamental groups of the list of Theorem
\ref{thm:fund-groups-complexity-4}
actually occur as fundamental groups of manifolds with Montesinos complexity
four. We have examples of filling Johansson diagrams with 4
triple points whose diagram groups are $\Z$, $\Z\oplus\Z$ or $\Z_q$ with $q\leq
5$, but we
have found no examples for $\Z_6$, $G_1$ or $G_2$.

The definition of filling Dehn spheres is naturally extended to \emph{filling
Dehn surfaces}, which are arbitrary compact immersed surfaces verifying the
conditions \textbf{(F1)}, \textbf{(F2)} and \textbf{(F3)} of Section
\ref{sec:Dehn-spheres-Johansson-diagrams}.
If we require the Dehn surface to be an immersed orientable surface of genus
$g$, we can talk about \emph{genus $g$ filling Dehn surfaces}.
In \cite{Rthesis} it is defined the \emph{triple point spectrum}
of a $3$-manifold $m$ as the sequence
$$\mathscr{T}(M)=(t_0(M),t_1(M),t_2(M)\ldots)\,,$$
where for all $g=0,1,2,\ldots$ the number $t_g(M)$ is the \emph{genus $g$
triple point number} of $M$, i.e. the minimal number of triple points among all
genus $g$ filling Dehn surfaces of $M$. Note that $mc(M)=t_0(M)$. A simple
surgery operation shows that the genus $g$ triple point numbers verify the
inequality $t_{g+1}(M)\leq t_g(M)+2$ for all $g=0,1,2,\ldots\,$, but the
equality does not necessarily hold because, for example, there are filling Dehn
tori with
just one triple point (see \cite{RHomotopies,Rthesis}). Apart from
this inequality, nothing is known about the triple point spectrum of any
$3$-manifold. A first question to answer in this context is if the triple
point spectrum of $S^3$ is $(2,4,6,\ldots)$.

All these numbers, as Amendola's surface-complexity, can be used to
give a census of
$3$-manifolds with increasing complexity. It should be interesting to
investigate if it can be designed an efficient computer program for giving a
list of $3$-manifolds with bounded Montesinos complexity,
as it has
been done for the Matveev complexity \cite{Matveev}, for example.

\end{document}